\documentclass[12pt]{amsart}
\title{Galois groups of large simple fields}
\author{Anand Pillay}
\address{University of Notre Dame}
\email{Anand.Pillay.3@nd.edu}
\author{Erik Walsberg}
\address{University of California, Irvine}
\email{ewalsber@uci.edu}
\urladdr{math.uci.edu/~ewalsber/}
\makeatletter
\let\@wraptoccontribs\wraptoccontribs
\makeatother
\contrib[with an appendix by]{Philip Dittmann}
\address{Technische Universität Dresden}
\email{philip.dittmann@tu-dresden.de}

\usepackage{amsmath, amssymb, amsthm}    
\usepackage{enumitem}
\usepackage[Symbol]{upgreek}
\usepackage{fullpage} 	
\usepackage{xcolor}

\DeclareFontFamily{U}{BOONDOX-calo}{\skewchar\font=45 }
\DeclareFontShape{U}{BOONDOX-calo}{m}{n}{
  <-> s*[1.05] BOONDOX-r-calo}{}
\DeclareFontShape{U}{BOONDOX-calo}{b}{n}{
  <-> s*[1.05] BOONDOX-b-calo}{}
\DeclareMathAlphabet{\mathcalboondox}{U}{BOONDOX-calo}{m}{n}
\SetMathAlphabet{\mathcalboondox}{bold}{U}{BOONDOX-calo}{b}{n}
\DeclareMathAlphabet{\mathbcalboondox}{U}{BOONDOX-calo}{b}{n}

\DeclareMathOperator*{\forkindep}{\raise0.2ex\hbox{\ooalign{\hidewidth$\vert$\hidewidth\cr\raise-0.9ex\hbox{$\smile$}}}}

\newcommand{\sopinfty}{\mathrm{SOP}_\infty}
\newcommand{\nsopinfty}{\mathrm{NSOP}_\infty}

\newcommand{\pac}{\mathrm{PAC}}

\newcommand{\kalg}{K^{\mathrm{alg}}}

\newcommand{\Spec}{\operatorname{Spec}}

\newcommand{\acl}{\operatorname{acl}}

\newcommand{\tp}{\operatorname{tp}}

\newcommand{\Chara}{\operatorname{Char}}
\newcommand{\jac}{\operatorname{Jac}}

\newcommand{\br}{\operatorname{Br}}

\newtheorem*{claim-star}{Claim}
\newtheorem{theorem}{Theorem}[section] 
\newtheorem{lemma}[theorem]{Lemma}

\newtheorem{prop-def}[theorem]{Proposition-Definition}
\newtheorem{corollary}[theorem]{Corollary}
\newtheorem{fact}[theorem]{Fact}
\newtheorem{fact-eh}[theorem]{Fact(?)}

\newtheorem{conjecture}[theorem]{Conjecture}

\newtheorem{proposition}[theorem]{Proposition}
\newtheorem{proposition-eh}[theorem]{Proposition(?)}
\newtheorem*{theorem-star}{Theorem}
\newtheorem*{conjecture-star}{Conjecture}
\newtheorem*{lemma-star}{Lemma}

\theoremstyle{definition}

\newtheorem{remark}[theorem]{Remark}
\theoremstyle{remark}
\newtheorem{claim}[theorem]{Claim}

\newcommand{\mfrakv}{\mathfrak{m}_v}
\newcommand{\Aa}{\mathbb{A}}

\newcommand{\Gg}{\mathbb{G}}
\newcommand{\Qq}{\mathbb{Q}}

\newcommand{\Zz}{\mathbb{Z}}

\newcommand{\nsopone}{\mathrm{NSOP}_1}

\newcommand{\meno}{\medskip \noindent}
\newcommand{\veno}{\vspace{2mm}\noindent}

\begin{document}
\maketitle

\begin{abstract}
Suppose that $K$ is an infinite field which is large (in the sense of Pop ~\cite{pop-embedding}) and whose first order theory is simple.
We show that $K$ is {\em bounded}, namely has only finitely many separable extensions of any given finite degree. 
We also show that any genus $0$ curve over $K$ has a $K$-point and if $K$ is additionally perfect then $K$ has trivial Brauer group.
These results give evidence towards the conjecture that large simple fields are bounded PAC.
Combining our results with a theorem of Lubotzky and van den Dries we show that there is a bounded $\mathrm{PAC}$ field $L$ with the same absolute Galois group as $K$.
In the appendix we show that if $K$ is large and $\nsopinfty$ and $v$ is a non-trivial valuation on $K$ then $(K,v)$ has separably closed Henselization, so in particular the residue field of $(K,v)$ is algebraically closed and the value group is divisible.
The appendix also shows that formally real and formally $p$-adic fields are $\sopinfty$ (without assuming largeness).
\end{abstract}

\section{Introduction}
Throughout $K$ is a field.
Large fields were introduced by Pop ~\cite{pop-embedding}, one definition being that $K$ is \textbf{large} if any algebraic curve defined over $K$ with a smooth (nonsingular) $K$-point has infinitely many $K$-points.
Finite fields, number fields, and function fields are not large.
Local fields, Henselian fields, quotient fields of Henselian domains, real closed fields, separably closed fields, pseudofinite fields, infinite algebraic extensions of finite fields, and fields which satisfy a local-global principle (in particular pseudo real closed and pseudo $p$-adically closed fields) are all large.
All infinite fields whose first order theory is known to be ``tame" or well-behaved in various senses, are large.
Let $K^{\mathrm{sep}}$ be a separable closure of $K$.
We say that $K$ is \textbf{bounded} if for any $n$ there are only finitely many degree $n$ extensions of $K$ in $K^\mathrm{sep}$, equivalently if the absolute Galois group  $\mathrm{Aut}(K^{\mathrm{sep}}/K)$ of $K$ has only finitely many open subgroups of any given finite index.
When $K$ is also perfect, this is also called Serre's property (F). 
(Other authors use ``bounded" to mean that $K$ has only finitely many extensions of each degree.)
Koenigsmann has conjectured that bounded fields are large~\cite[p.~496]{JK}.

\veno
Recall that $K$ is \textbf{pseudoalgebraically closed ($\mathrm{PAC}$)} if any geometrically integral $K$-variety $V$ has a $K$-point (and hence the set $V(K)$ of $K$-points is Zariski dense in $V$).
We mention in passing that a $\mathrm{PAC}$ field need not be perfect, if $p = \Chara(K)$, and $a\in K$ is not a $p$th power, then $\Spec K[x]/(x^p - a)$ is not geometrically integral, see \cite[Example 2.2.9]{poonen-qpoints}.
$\mathrm{PAC}$ fields are large, by definition.
$\mathrm{PAC}$ fields were introduced by Ax~\cite{ax-finite} who showed that pseudofinite fields are bounded $\mathrm{PAC}$.
Infinite algebraic extensions of finite fields are also bounded $\mathrm{PAC}$, see~\cite[11.2.4]{field-arithmetic}.
In either case $\mathrm{PAC}$ follows from the Hasse-Weil estimates.

\veno
On the model theoretic side, we have various ``tame" classes of first order theories $T$, the most ``perfect" being stable theories, and some others being simple theories and $\mathrm{NIP}$ theories.
It is a well-known theorem of Shelah that a theory is stable if and only if it is both simple and $\mathrm{NIP}$.
Good examples come from theories of fields.
We say that a first order structure, in particular a field, is stable (simple, NIP) if its theory is stable (simple, $\mathrm{NIP}$).
Separably closed fields are stable and bounded $\mathrm{PAC}$ fields are simple.
There is a considerable amount of work on $\mathrm{NIP}$ fields, which include real closed and $p$-adically closed fields, but this will not concern us in the present paper.
We now recall two longstanding open conjectures.

\begin{conjecture}
\label{conj:main}\quad
\begin{enumerate}
\item Infinite stable fields are separably closed.
\item Infinite simple fields are bounded $\mathrm{PAC}$.
\end{enumerate}
\end{conjecture}

Our general idea is that Conjecture~\ref{conj:main} is both true and tractable after making the additional assumption of largeness. 
It is shown in \cite{JTWY} that a large stable field is separably closed.
We describe another proof of this result in Section~\ref{section:another-proof}.
Here we consider (2), and prove:

\begin{theorem}
\label{thm:main}
Suppose that $K$ is large and simple.
Then there is a bounded $\mathrm{PAC}$ field $L$ of the same characteristic as $K$ such that the absolute Galois group of $L$ is isomorphic (as a topological group) to the absolute Galois group of $K$.
\end{theorem}

Note that Theorem~\ref{thm:main} implies that $K$ is bounded.
We prove this separately.


\begin{theorem}
\label{thm:bounded}
If $K$ is large and simple then $K$ is bounded.
\end{theorem}

The assumption that $K$ is simple can be replaced by the more general assumption that the field $K$ is definable in some model $M$ of a simple theory.  If we also require $M$ to be highly saturated we can take $K$ to be type-definable (over a small set of parameters) in $M$.
The latter will follow from our proofs and references and we will not talk about it again. 
Theorem~\ref{thm:bounded} generalizes the theorem of Chatzidakis that a simple $\mathrm{PAC}$ field is bounded, which is proven via quite different methods in \cite{Chatzidakis}.
Poizat ~\cite{Poizat-generic} proved that an infinite stable bounded field is separably closed.
Combining Poizat's result with Theorem \ref{thm:bounded} we get the above mentioned result of ~\cite{JTWY} that large stable fields are separably closed.

\veno
Theorem~\ref{thm:bounded} is reasonably sharp.
The restriction to separable extensions is necessary.
If $K$ is separably closed of infinite imperfection degree and $\Chara(K) = p > 0$ then $K$ is large, stable, and has infinitely many extensions of degree $p$.
There is an emerging body of work on a generalization of simplicity known as $\mathrm{NSOP_1}$.
All known $\nsopone$ fields are $\pac$.
Theorem~\ref{thm:bounded} fails over $\mathrm{NSOP_1}$ fields as there are unbounded $\mathrm{PAC}$ $\mathrm{NSOP_1}$ fields (equivalently: there are $\mathrm{PAC}$ fields that are $\mathrm{NSOP}_1$ but not simple).
For example if $K$ is characteristic zero, $\mathrm{PAC}$, and the absolute Galois group of $K$ is a free profinite group on $\aleph_0$ generators then $K$ is unbounded and $\mathrm{NSOP}_1$~\cite[Corollary 6.2]{cr-tree}.

\veno
A profinite group $G$ is \textbf{projective} if any continuous surjective homomorphism $H \to G$, $H$ profinite, has a section.
Ax showed that the absolute Galois group of a perfect $\mathrm{PAC}$ field is projective~\cite{ax-finite}, Jarden proved this for non-perfect $\pac$ fields~\cite[Lemma 2.1]{jarden-projective}.

\begin{theorem}
\label{thm:projective}
If $K$ is large and simple then the absolute Galois group of $K$ is projective.
\end{theorem}

Theorem~\ref{thm:main} follows from Theorem~\ref{thm:bounded}, Theorem~\ref{thm:projective}, and the theorem of Lubotzky and van den Dries that for any field $K$ and projective profinite group $G$ there is a $\pac$ field extension $L$ of $K$ such that the absolute Galois group of $L$ is isomorphic to $G$, see \cite[Corollary 23.1.2]{field-arithmetic}.
An earlier version of this paper proved Theorem~\ref{thm:projective} under the additional assumption that $K$ is perfect.
Philip Dittmann showed us how to remove this assumption.

\begin{theorem}
\label{thm:brauer}
Suppose that $K$ is perfect, large, and simple.
Then the Brauer group of $K$ is trivial.
It follows that
\begin{enumerate}
\item any finite dimensional division algebra over $K$ is a field, and
\item any Severi-Brauer $K$-variety $V$ has a $K$-point.
\end{enumerate}
\end{theorem}

We recall the definition of Severi-Brauer variety.
Let $\kalg$ be an algebraic closure of $K$.
Given a $K$-variety $V$ we let $V_{\kalg}$ be the base change $V \times_K \Spec \kalg$ of $V$ to a $\kalg$-variety.
A \textbf{Severi-Brauer variety} is a $K$-variety $V$ such $V_{\kalg}$ is isomorphic (over $\kalg$) to $\dim V$-dimensional projective space.
A Severi-Brauer variety is geometrically integral, so $(2)$ is a modest step towards the conjecture that large simple fields are $\mathrm{PAC}$.
Theorem~\ref{thm:brauer} was proven for supersimple fields in \cite{supersimple-field}, our proof closely follows that in \cite{supersimple-field}, so we will not recall the definition of the Brauer group.
(Supersimple fields are perfect, but large simple fields need not be perfect.)
Items $(1)$ and $(2)$ of Theorem~\ref{thm:brauer} are well-known consequences of triviality of the Brauer group.
We refer to \cite[1.5, 4.5.1]{poonen-qpoints} for the definition of the Brauer group and these facts.

\veno
We describe the proof of Theorem~\ref{thm:projective} under the additional assumption of perfection.
It is a theorem of Gruenberg~\cite{gruenberg} that a field of cohomological dimension $\le 1$ has projective absolute Galois group, and if every finite extension of $K$ has trivial Brauer group then $K$ has cohomological dimension $\le 1$ ~\cite[II.3.1, Proposition 5]{serre-coho}.
The class of perfect fields is closed under finite extensions.
If $K$ is simple then any finite extension $L$ of $K$ is simple as $L$ is interpretable in $K$.
Finally, a finite extension of a large field is large by Fact~\ref{fact:alg-ext} below.
So Theorem~\ref{thm:brauer} implies that every finite extension of a perfect large simple field has trivial Brauer group.

\vspace{2mm}
\noindent
Suppose $\Chara(K) \ne 2$, then we say that a \textbf{conic} over $K$ is a smooth irreducible projective $K$-curve of genus $0$.
One-dimensional Severi-Brauer varieties are exactly conics~\cite[4.5.8]{poonen-qpoints}.
Thus Theorem~\ref{thm:conics} generalizes the one-dimensional case of Theorem~\ref{thm:brauer}.2 to imperfect fields.

\begin{theorem} 
\label{thm:conics}
Suppose that $K$ is large and simple, $\Chara(K) \ne 2$, and $C$ is a conic over $K$.
Then $C$ has a $K$-point, hence (by largeness) $C(K)$ is infinite.
\end{theorem}

Let us mention some other earlier work around the conjectures on stable and simple fields described above.
One of the first results on deducing algebraic results from model-theoretic hypotheses was Macintyre's theorem that infinite fields with $\upomega$-stable theory are algebraically closed (\cite{Macintyre-omegastable} and generalized to superstable fields in ~\cite{Cherlin-Shelah-superstable}).  
Macintyre's Galois-theoretic method has been used in many later works including the result on large stable fields ~\cite{JTWY} mentioned above.
Supersimple theories are simple theories in which there are not infinite forking chains of types, whereby any complete type has an ordinal valued dimension called the $SU$-rank.  This gives a so-called ``surgical dimension" as in ~\cite{corps-et-ch} from which one deduces that an infinite field with {\em supersimple} theory is perfect and bounded.  So in so far as Conjecture (2) is restricted to supersimple theories, it remained to prove that supersimple theories are $\mathrm{PAC}$, and some partial results were obtained in ~\cite{supersimple-field} and ~\cite{Amador-elliptic} for example.
A theme of the current paper is that, other than perfection of $K$, any results on supersimple fields also hold over large simple fields.

\meno
The conclusions of Theorems~\ref{thm:projective}, ~\ref{thm:brauer}, and \ref{thm:conics} are properties of $\mathrm{PAC}$ fields.
Another well-known consequence of a field $K$ being PAC is that the Henselization of any non-trivial valuation on $K$ is separably closed, see \cite[Corollary 11.5.9]{field-arithmetic}.
In an earlier draft of this paper we showed that any non-trivial valuation on a large simple field has separably closed Henselization.
Dittmann generalized this to Theorem~\ref{thm:sop}, which is proven in the appendix.

\begin{theorem}
\label{thm:sop}
Suppose that $K$ is large and $\nsopinfty$.
Then any non-trivial valuation on $K$ has separably closed Henselization.
In particular any nontrivial valuation on $K$ has algebraically closed residue field and divisible value group.
\end{theorem}

$\nsopinfty$ is a weakening of simplicity, see the appendix for a definition and some discussion.
$\nsopone$ implies $\nsopinfty$ and essentially every known theory without the strict order property is $\nsopinfty$.
It is natural to ask if Theorem~\ref{thm:sop} holds without the assumption of largeness.
In the appendix we give the following partial generalization.

\begin{theorem}
\label{thm:formally p adic}
If $K$ admits a $p$-valuation then $K$ is $\sopinfty$.
In particular if $K$ is a subfield of a finite extension of $\Qq_p$ then $K$ is $\sopinfty$.
\end{theorem}

See Section~\ref{section:formally real} for the definition of a $p$-valuation.
The proof of Theorem~\ref{thm:formally p adic} uses diophantine work of Anscombe, Dittmann, and Fehm~\cite{AnscombeDittmannFehm-Siegel} in place of largeness.
The results of \cite{AnscombeDittmannFehm-Siegel} are $p$-adic analogues of classical results on sums of squares.
In Section~\ref{section:formally real} we give a similar argument using Lagrange's four square theorem to show that a formally real field is $\sopinfty$.
If $K$ admits a valuation with formally real residue field then $K$ is formally real~\cite[Corollary 10.1.9]{real-algebraic-geometry}.
Thus if $K$ admits a valuation with formally real residue field then $K$ is $\sopinfty$.

\subsection*{Acknowledgements}
The first author was supported by NSF grants DMS-1665035, DMS-1760212, and DMS-2054271.
We thank Daniel Max Hoffmann for finding some mistakes.



\section{Large fields and definability}

\subsection{Algebraic conventions}
We let $K^*$ be the set of non-zero elements of $K$ and $\Chara(K)$ be the characteristic of $K$.
A \textbf{$K$-variety} is a separated, reduced, $K$-scheme of finite type.
We let $\dim V$ be the usual algebraic dimension and $V(K)$ be the set of $K$-points of a $K$-variety $V$.
We let $\Aa^n$ be $n$-dimensional affine space over $K$, recall that $\Aa^n(K) = K^n$.
We will often assume irreducibility of the relevant $K$-varieties. 
A $K$-curve is a one-dimensional $K$-variety.
A \textbf{morphism} is a morphism of $K$-varieties.

\subsection{Largeness}
Large fields were introduced by Florian Pop. A survey appears in ~\cite{Pop-little} which starts by saying that large fields are fields over which (or in which) one can do a lot of ``interesting mathematics".  So largeness looks like a field-arithmetic  tameness notion. 
The field $K$ is \textbf{large} if every irreducible $K$-curve with a smooth (also called nonsingular)  $K$-point has infinitely many $K$-points. 
Fact~\ref{fact:pop-def} is due to Pop~\cite{pop-embedding}.

\begin{fact}
\label{fact:pop-def}
The following are equivalent:
\begin{enumerate}[leftmargin=*]
\item $K$ is large,
\item $K$ is existentially closed in $K((t))$,
\item if an irreducible $K$-variety $V$ has a smooth $K$-point then $V(K)$ is Zariski dense in $V$.
\end{enumerate}
\end{fact}

We also make use of Fact~\ref{fact:alg-ext}, see \cite[Proposition 2.7]{Pop-little}.

\begin{fact}
\label{fact:alg-ext}
An algebraic extension of a large field is large.
\end{fact}

Fact~\ref{fact:elem-ext} allows us to pass to elementary extensions, see \cite[Proposition 2.1]{Pop-little}.

\begin{fact}
\label{fact:elem-ext}
Large fields form an elementary class.
\end{fact}

\subsection{Existentially \'etale sets}
Let $W$ be a $K$-variety.
The authors of \cite{JTWY} introduced the \textbf{\'etale open topology} on $W(K)$.
If $K$ is not large then the \'etale open topology is always discrete and if $K$ is large then the \'etale open topology on $W(K)$ is non-discrete whenever $W(K)$ is infinite.
Our original proofs were given in terms of this topology, but at present we will mostly avoid the topology and give proofs from scratch.
We will use properties of certain special existentially definable subsets of $W(K)$.
A subset $X$ of $W(K)$ is an \textbf{EE} set if there is a $K$-variety $V$ and an \'etale morphism $f :V \to W$ such that $X = f(V(K))$.
It is shown in \cite{JTWY} that the EE subsets of $W(K)$ form a basis for the \'etale open topology.
(In \cite{JTWY} EE sets are referred to as ``\'etale images".)

\veno
If $W$ is smooth and $V \to W$ is an \'etale morphism then $V$ is also smooth.
At present we are mainly concerned with subsets of $K^n = \Aa^n(K)$, so we may restrict attention to smooth $K$-varieties.
We quickly recall what we need from this setting.
Let $V,W$ be smooth irreducible $K$-varieties. An \textbf{\'etale morphism} $f:V\to W$ is a morphism such that the differential $df_{a}$ is an isomorphism $TV_a \to TW_{f(a)}$ for all $a \in V$.
In particular if $f \colon \Aa^n \to \Aa^n$ is a morphism then $f$ is \'etale at $a \in K^n$ if and only if the Jacobian of $f$ at $a$ is invertible.
The general notion of an \'etale morphism between not necessarily smooth varieties is more complicated but will not be needed here.
Fact~\ref{fibre-product} is proven in \cite[Proposition 17.1.3]{EGA-IV-4}.

\begin{fact}
\label{fibre-product}
Suppose $W_{1}, W_{2}, V$ are smooth $K$-varieties and $f_{i}:W_{i}\to V$ is an \'etale morphism for $i \in \{1,2\}$.
Let $W$ be the fibre product $W_1 \times_{V} W_{2}$ and  $f:W\to V$ be the canonical map. 
Then $W$ is a smooth $K$-variety and $f$ is \'etale. 
\end{fact}

We have $(W_1 \times_V W_2)(K) = \{ (a_1,a_2) \in W_1(K) \times W_2(K) : f_1(a_1) = f_2(a_2) \}$.
It easily follows that the image of $(W_1 \times_V W_2) (K)$ under $f$ agrees with $f_1(W_1(K)) \cap f_2(W_2(K))$.
Corollary~\ref{cor:intersection} follows.

\begin{corollary}
\label{cor:intersection}
Suppose that $W$ is a smooth $K$-variety.
Then the collection of EE subsets of $W(K)$ is closed under finite intersections.
\end{corollary}

Corollary~\ref{cor:intersection} holds for an arbitrary $K$-variety, but we will not need this.

\begin{lemma}
\label{Zariski-dense}
Suppose that $K$ is large, $W$ is a smooth irreducible $K$-variety, and $X$ is a nonempty EE subset of $W(K)$.
Then $X$ is Zariski-dense in $W$.
In particular any nonempty EE subset of $K^n$ is Zariski dense in $K^n$.
\end{lemma}

The identity morphism $W \to W$ is \'etale, so Lemma~\ref{Zariski-dense} generalizes the fact that if $K$ is large and $W$ is a smooth irreducible $K$-variety with $W(K) \ne \emptyset$ then $W(K)$ is Zariski dense in $W$.

\begin{proof}
Let $V$ be a $K$-variety and $f \colon V \to W$ be an \'etale morphism such that $X = f(V(K))$.
Suppose that $X$ is not Zariski dense in $V$.
Then $X$ is contained in a proper closed subvariety $Y$ of $W$.
As $W$ is irreducible we have $\dim Y < \dim W$.
Note that $f^{-1}(Y)$ is a closed subvariety of $V$ containing  $V(K)$.
As $f$ is \'etale it is finite-to-one, hence $\dim V = \dim W$ and $\dim f^{-1}(Y) = \dim Y < \dim W$.
So $f^{-1}(Y)$ is a proper closed subvariety of $V$ containing $V(K)$.
This contradicts Fact~\ref{fact:pop-def}.
\end{proof}

Corollary~\ref{corollary:intersection} follows from Corollary~\ref{cor:intersection} and Lemma~\ref{Zariski-dense}.

\begin{corollary} \label{corollary:intersection} 
Suppose that $K$ is large.
Let $W$ be a smooth irreducible $K$-variety and $X_1,\ldots,X_n$ be EE subsets of $W(K)$ with $\bigcap_{i = 1}^{k} X_i \ne \emptyset$.
Then $\bigcap_{i = 1}^{k} X_i$ is Zariski dense in $W$.
In particular if $X_1,\ldots,X_k$ are EE subsets of $K^n$ with $\bigcap_{i = 1}^{k} X_i \ne \emptyset$ then $\bigcap_{i = 1}^{k} X_i$ is Zariski dense in $K^n$.
\end{corollary}

Fact~\ref{remark:translation} is proven in \cite{JTWY} for arbitrary $K$-varieties.

\begin{fact}
\label{remark:translation} 
Let $W$ be a smooth $K$-variety, $g \colon W \to W$ be a $K$-variety isomorphism, and $X$ be an EE subset of $W(K)$.
Then $g(X)$ is also an EE subset of $W(K)$. 
\end{fact}

\begin{proof}
Let $V$ be a smooth $K$-variety and $f \colon V \to W$ be an \'etale morphism such that $X = f(V(K))$.
Note that $g$ is \'etale as any $K$-variety isomorphism is \'etale.
So $g \circ f \colon V \to W$ is \'etale as a composition of \'etale morphisms is \'etale.
\end{proof}

We will apply Corollary~\ref{cor:translate} below.

\begin{corollary}
\label{cor:translate}
Suppose that $X$ is an EE
subset of $K^n$, $a = (a_1,\ldots,a_n) \in K^n$, and $b = (b_1,\ldots,b_n) \in (K^*)^n$.
Then
$$ X + a = \{ (c_1 + a_1,\ldots,c_n + a_n) : (c_1,\ldots,c_n) \in X \} $$
and
$$ bX = \{ (b_1 c_1,\ldots,b_n c_n) : (c_1, \ldots, c_n) \in X \}$$
are EE subsets of $K^n$.
\end{corollary}

\begin{proof}
The morphisms $\Aa^n \to \Aa^n$ given by $(x_1,\ldots,x_n) \mapsto (x_1 + a_1,\ldots,x_n + a_n)$ and also by $(x_1,\ldots,x_n) \mapsto (b_1 x_1,\ldots,b_n x_n)$ are $K$-variety isomorphisms.
Apply Fact~\ref{remark:translation}.
\end{proof}

\section{Fields with simple theory}
We recall some basic results about fields $K$ whose first order theory is simple, and then make an additional observation under the assumption of largeness.
For simple theories see ~\cite{Kim-Pillay} and ~\cite{Casanovas}, and for groups definable in (models of) simple theories, see   in addition ~\cite{Pillay-definability-simple} and ~\cite{supersimple-field}.
We recall the relevant portions of this theory.

\subsection{Conventions and basic definitions}
Our model theoretic notation is standard. 
We let $L$ be a first order language, $T$ be a complete consistent $L$-theory, and $\overline{M}$ be a highly saturated model of $T$.
For now, $x, y, z, \ldots $ range over finite tuples of variables, $a,b,c,\ldots$ range over finite tuples of parameters from $\overline{M}$, and $A,B,C,\ldots$ range over small subsets of $\overline{M}$.
``Definable" means ``definable in $\overline{M}$, possibly with parameters".
We will sometimes identify definable sets with the formulas defining them. 

\veno
Given an $L$-formula $\phi(x,y)$ and a suitable tuple $b$ we say that $\phi(x,b)$ \textbf{divides over} a set $A$ of parameters if $\{\phi(x,b_{i}):i< \upomega\}$ is inconsistent for some infinite $A$-indiscernible sequence $(b_{i}:i<\upomega)$ with $b_{0} = b$.
A partial type $\Sigma(x)$ divides over $A$ if some formula in $\Sigma$ divides over $A$. 
The theory $T$ is
\textbf{simple} if for any small set $A$ of parameters and complete type $\Sigma(x)$ there is $A_0 \subseteq A$ such that $|A_0| \le |T|$ and $\Sigma(x)$  \textbf{does not divide} over $A_{0}$.
Simplicity may also be defined in terms of the combinatorial \textbf{tree property}, but we will not need this.
It is worth mentioning that simplicity is incompatible with the existence of a definable partial ordering which contains an infinite chain.
It follows that real closed fields and non-separably closed Henselian fields are not simple.
Non-dividing yields a good notion of independence in simple theories: $a$ is \textbf{independent} from $B$ over $A$ if $\tp(a/B,A)$ does not divide over $A$.

\subsection{Generics in definable groups}
In this section we summarize  \cite[Section 3]{Pillay-definability-simple}, although we introduce things in a different order and use somewhat different terminology.
Suppose that $T$ is simple and $G$ is an infinite group definable over $\emptyset$ in $\overline{M}$.
A definable subset $X$ of $G$  is \textbf{(left) $f$-generic} if every left translate $gX$ of $X$ does not divide over $\emptyset$ and a complete type $\Sigma(x)$ concentrated on $G$ is (left) $f$-generic if every formula in $\Sigma(x)$ is left $f$-generic.
Note that if a definable $X \subseteq G$ is $f$-generic then $aX$ is $f$-generic for any $a \in G$.
Note that in \cite{Pillay-definability-simple} ``generic" is used for ``$f$-generic".
(The language was changed after some more recent work on groups definable in NIP theories.)

\begin{fact}
\label{fact:generic}
Suppose that $T$ is simple, $G$ is an $\emptyset$-definable group in $\overline{M}$, $A,B$ are small sets of parameters, and $a \in G$.
Then
\begin{enumerate}
\item Left $f$-genericity is equivalent to right $f$-genericity, so we just say $f$-generic.
\item If $X \subseteq G$ is $f$-generic then $X$ is $f$-generic in any expansion of $\overline{M}$ by constants.
\item $\tp(a/A)$ is left $f$-generic if whenever $b\in G$ is independent from $a$ over $A$ then the product $ba$ is independent of $A\cup\{b\}$ over $\emptyset$.
 \item if $A \subseteq B$ and $a$ is independent from $B$ over $A$, then $\tp(a/B)$ is $f$-generic if and only if $\tp(a/A)$ is $f$-generic.
\item if $b \in B$ then $\tp(a/A,b)$ is $f$-generic if and only if $\tp(ba/A,b)$ is $f$-generic.
\item an $A$-definable subset $X$ of $G$ is $f$-generic if and only if it is contained in an $f$-generic complete type over $A$.
\end{enumerate}
\end{fact}

Fact~\ref{fact:intersect} is immediate from the definitions and Fact~\ref{fact:generic}.

\begin{fact}
\label{fact:intersect}
If $X$ is a definable subset of $G$ which is not $f$-generic then we have $\bigcap_{i = 1}^{k} g_iX = \emptyset$ for some $g_1,\ldots,g_k \in G$.
\end{fact}

\begin{lemma}
\label{lem:index-generic}
Suppose that $T$ is simple, $M$ is a model of $T$, $G$ is an $\emptyset$-definable group in $M$, $H$ is a subgroup of $G$ with $|G/H| \ge \aleph_0$, and $X$ is a definable subset of $G$ such that $X \subseteq aH$ for some $a \in G$.
Then $X$ is not $f$-generic.
In particular an infinite index definable subgroup of $G$ is not $f$-generic.
\end{lemma}

\begin{proof}
Let $(g_i : i < \upomega)$ be a sequence of elements of $G$ which lie in distinct cosets of $H$.
So $g_i X \cap g_j X = \emptyset$ when $i \ne j$.
After passing to a highly saturated elementary extension and applying Ramsey and saturation we obtain an sequence $(h_i : i < \upomega)$ of elements of $G$ which is indiscernible over the defining parameters of $X$ and satisfies $h_i X \cap h_j X = \emptyset$ when $i \ne j$.
So $X$ is not $f$-generic.
\end{proof}

\begin{lemma}
\label{lem:equiv-rel}
Suppose that $T$ is simple, $X$ is a definable subset of $G$, $\approx$ is a definable equivalence relation on $X$, and each $\approx$-class is $f$-generic.
Then there are only finitely many $\approx$-classes.
\end{lemma}

\begin{proof}
Suppose towards a contradiction that there are infinitely many $\approx$-classes.
Let $c$ be a finite tuple of parameters over which $X$ and $\approx$ are definable.
Then there is an $\approx$-class $D$ with canonical parameter $d$ such that $d\notin \acl(c)$. 
Let $\phi(x,d,c)$ be a formula defining $D$ and $(d_{i}:i<\upomega)$ be an infinite sequence of realizations of $\tp(d/c)$ which is indiscernible over $c$ and satisfies $d_{0} = d$.
Then $( (c,d_{i}):i<\upomega)$ is indiscernible, and the formulas $\phi(x,d_{i},c)$ are pairwise inconsistent, so $\phi(x,d,c)$ divides over $\emptyset$. This contradicts that $\phi(x,d,c)$ defines the set $D$ which is an $f$-generic subset of $K^{n}$.
\end{proof}

We now prove Lemma~\ref{fact:generic-product}, which we could not find in the literature.
\begin{lemma}
\label{fact:generic-product}  
Suppose $T$ is simple and $G, H$ are $\emptyset$-definable groups  in $\overline{M}$.
Fix a small set $A$ of parameters and $(a,b)\in G\times H$.
Then $\tp((a,b)/A)$ is $f$-generic in $G\times H$ if and only if the following conditions hold:
\begin{enumerate}
\item $\tp(a/A)$ is an $f$-generic type of $G$, 
\item $\tp(b/A)$ is an $f$-generic type of $H$,
\item and $a$ is independent from $b$ over $A$.
\end{enumerate}
\end{lemma}

\begin{proof}
The definitions and ``forking calculus" easily show that $(1),(2),$ and $(3)$ together imply the $\tp((a,b)/A)$ is $f$-generic in $G \times H$.
The difficulty lies in showing that all $f$-generic types of $G\times H$ are of this form.
We suppose that $\tp((a,b)/A)$ is $f$-generic in $G\times H$. 
It follows directly that $\tp(a/A)$, and $\tp(b/A)$ are $f$-generic types of $G, H$ respectively.  
It remains to prove that $a$ is independent from $b$ over $A$.
Suppose that $(c,d) \in G \times H$, $\tp(c/A), \tp(d/A)$ is $f$-generic in $G,H$, respectively, and $(c,d)$ is independent from $(a,b)$ over $A$.
By Fact~\ref{fact:generic} $ca$ is independent from $db$ over $\emptyset$.
As $\tp((a,b)/A)$ is $f$-generic in $G\times H$, and $(a,b)$ is independent from $(c,d)$ over $A$, we see that that $(ca,db)$ is independent from $A,c,d$ over $\emptyset$.
It follows that $a$ is independent from $b$ over $A,c,d$, and then that $a$ is independent from $b$ over $A$. 
\end{proof}

\subsection{Generics in definable fields}
Now suppose $K$ is an infinite field definable (say over $\emptyset$) in $\overline{M}\models T$.
Everything we say remains true for $K$ a type-definable field in $\overline{M}$. 
We have two attached groups, the additive group $(K,+)$ and the multiplicative group $(K^{*},\times)$, recall that $K^{*} = K\setminus \{0\}$.
A definable $X \subseteq K$ is \textbf{additively $f$-generic} if it is $f$-generic in $(K,+)$ and is \textbf{multiplicatively $f$-generic} if $X \cap K^*$ is an $f$-generic in $(K^*,\times)$, and we make the analogous definitions for a type concentrated on $K$.
The first claim of Fact~\ref{fact:generic-basic} is \cite[Proposition 3.1]{supersimple-field}.
Uniqueness of $f$-generic types in stable fields is \cite[Theorem 5.10]{Poizat}.

\begin{fact}
\label{fact:generic-basic}
Suppose that $T$ is simple.
Let $X$ be a definable subset of $K$, $A$ be a small set of parameters, and $p = \tp(a/A)$ for some $a \in K$.
Then the following are equivalent:
\begin{enumerate}
\item $X$ is an additive $f$-generic.
\item $X$ is a multiplicative $f$-generic.
\end{enumerate}
Furthermore the following are equivalent:
\begin{enumerate}
\item $p$ is an additive $f$-generic \item $p$ is a multiplicative $f$-generic.
\end{enumerate}
If $T$ is stable then there is a unique additive $f$-generic type over $K$. 
\end{fact}

We let $D_n$ be the group $((K^*)^n, \times)$. 
Corollary~\ref{corollary:generic} is a higher dimensional version of Fact~\ref{fact:generic-basic}.
The first claim of Corollary~\ref{corollary:generic} follows from Fact~\ref{fact:generic-basic},  Lemma~\ref{fact:generic-product}, and induction on $n$.
The second claim follows from the first claim and Fact~\ref{fact:generic}.5.

\begin{corollary} \label{corollary:generic} 
Suppose that $T$ is simple, $A$ is a small set of parameters, $a = (a_{1},..,a_{n})\in K^{n}$, and $p(x) = \tp(a/A)$.
Then $p$ is an $f$-generic type of $(K^{n},+)$ if and only if $p$ is an $f$-generic type of $D_{n}$.
So if $X\subseteq K^{n}$ is definable, then $X$ is $f$-generic in $(K^{n},+)$ if and only if $X\cap D_{n}$ is $f$-generic in $D_{n}$.
\end{corollary}

Proposition~\ref{prop:generic-open} is our main tool when dealing with large simple fields.

\begin{proposition} 
\label{prop:generic-open}  
Suppose that $T$ is simple and $K$ is large.
Let $X$ be definable subset of $K^n$ which contains a nonempty EE subset.
Then $X$ is $f$-generic for $(K^{n},+)$, and is hence $f$-generic for $D_n$.
\end{proposition}

Thus if $T$ is simple and large then any definable subset of $K^n$ with nonempty interior in the \'etale open topology is $f$-generic.
If $K$ is bounded $\pac$ then a definable subset of $K^n$ is $f$-generic if and only if it is $f$-generic~\cite{secondpaper}. 

\begin{proof} 
Suppose towards a contradiction that $X$ is not $f$-generic for $(K^{n},+)$. 
By Corollary~\ref{corollary:generic}, $X\cap D_{n}$ is not $f$-generic for $D_{n}$.
We may suppose that $X$ contains $\overline{0} = (0,\ldots,0)$ as both EE subsets and $f$-generic subsets of $K^{n}$ are closed under additive translation  (by Corollary~\ref{cor:translate} and definitions). 
Let $X' = X\cap D_{n}$. By Corollary ~\ref{corollary:generic}, $X'$ is not $f$-generic in $D_{n}$.
By Fact~\ref{fact:intersect} there are $g_{1},..,g_{k}\in D_{n}$ such that $\bigcap_{i=1}^{k} g_{i}X' = \emptyset$. 
Then $\bigcap_{i=1}^{k}g_{i}X$ is nonempty, as it contains $\overline{0}$, but is contained in $K^{n}\setminus D^{n}$ and is hence not Zariski dense in $K^n$. 
This contradicts Corollary \ref{corollary:intersection}. 
\end{proof}

Fact~\ref{fact:coset-sum} will be crucial for Theorem~\ref{thm:brauer}.
It is proven in \cite{supersimple-field}.

\begin{fact}
\label{fact:coset-sum}
Suppose that $T$ is simple.
Let $H$ be a finite index definable subgroup of $(K^{*},\times)$ and $H_1, H_2$ be cosets of $H$.
Then $H_{1} + H_{2}$ contains $K^{*}$, namely every nonzero element of $K$ is of the form $a+b$ where $a\in H_{1}$ and $b\in H_{2}$.
\end{fact}

\section{Proof of Theorem ~\ref{thm:bounded}}
This section is the proof of Theorem~\ref{thm:bounded}.
Remember that when we say that $K$ is bounded we mean that for every $n$,  $K$ has only finitely many extensions of any given degree inside $K^{\mathrm{sep}}$.
We first make a few reductions.
Fact~\ref{fact:reduce} is well-known, we include a proof for the sake of completenes.

\begin{fact}
\label{fact:reduce}
The following are equivalent:
\begin{enumerate}
\item $K$ is bounded,
\item for any $n$ there are only finitely many degree $n$ separable extensions of $K$ up to $K$-algebra isomorphism.
\end{enumerate}
\end{fact}

\begin{proof}
By the primitive element theorem a degree $n$ separable extension  $L$  of $K$ is of the form $L = K(\alpha)$ where $\alpha$ is a root of a separable irreducible  monic degree $n$ polynomial $p(x) \in K[x]$.
So $L$ has at most $n$ distinct conjugates over $K$ in $K^{\mathrm{sep}}$, the fact easily follows.
\end{proof}

We set some notation.
Given $a = (a_{0},...,a_{n-1})\in K^{n}$ we let $p_{a}(x)$ denote the polynomial  $x^{n} + a_{n-1}x^{n-1} + \ldots + a_{1}x + a_{0}$.  
We let $U$ be the set of $a \in K^n$ such that $p_{a}$ is separable and irreducible in $K[x]$.
Note that $U$ is definable.
Given $p \in K[x]$ we let $(p)$ be the ideal in $K[x]$ generated by $p$.
For each $a \in U$ the field extension $K(\alpha)$ generated over $K$ by a root $\alpha$ of $p_{a}$ is isomorphic to $K[x]/(p_{a})$.
For $a,b\in U$, we write $a\approx b$ if $K[x]/(p_{a})$ is isomorphic over $K$ to $K[x]/(p_{b})$.
So $K$ has finitely many separable extensions of degree $n$ if and only if there are only finitely many $\approx$-classes.

\begin{remark} 
\label{remark:eq-def}
The equivalence relation $\approx$ on $U$ is definable in $K$.
\end{remark}

\begin{proof} 
The field $K[x]/(p_{a})$ is uniformly interpretable in $K$ (as $a$ varies), as an $n$-dimensional vector space over $K$ (with basis $1, \alpha, \ldots , \alpha^{n-1}$ for $\alpha$ a root of $p_{a}(x)$ and the appropriate multiplication).
Now note that if $a,b \in U$ then $a \approx b$ if and only if $p_{b}$ has a root in $K[x]/(p_{a})$.
\end{proof} 

The main result we have to prove to obtain Theorem ~\ref{thm:bounded} is:

\begin{theorem}
\label{thm:krasner} 
Suppose that $a \in U$ and let $D$ be the $\approx$-class of $a$.
Then there is an EE subset $X$ of $K^n$ such that $a \in X \subseteq D$.
\end{theorem} 

Equivalently: every $\approx$-class is \'etale open.
We now work towards the proof of Theorem ~\ref{thm:krasner}.

\veno
Fix $a\in U$, and let $\alpha \in K^{\mathrm{sep}}$ be a root of $p_{a}(x)$.
Let $\overline{x} = (x_0,\ldots,x_{n - 1})$ be a tuple of variables and let $\beta(\overline{x}) = x_0 + \alpha x_1 + \ldots + x_{n - 1} \alpha^{n - 1}$.
Let $\alpha = \alpha_{1},\ldots,\alpha_{n}$, be the $K$-conjugates of $\alpha$, namely the roots of $p_{a}(x)$ (which are distinct).  
We write $\beta_{i}(\overline{x})$ for
$x_{0} + x_{1}\alpha_{i} + \cdots + x_{n-1}\alpha_{i}^{n-1}$.
So, for $b \in K^{n}$, $\beta_{1}(b),\ldots,\beta_{n}(b)$ are the $K$-conjugates of $\beta(b)$. 

\veno
Let $V$ be the set of $b = (b_{0}, b_{1}, \ldots, b_{n-1})\in K^{n}$ such $K(\beta(b)) = K(\alpha)$.
Note that $b \in V$ if and only if $\beta(b)$ is a root of $p_{c}(x)$ for some (in fact unique) $c\in U$ such that $c \approx a$.
Note further that $b \in V$ if and only if $1,\beta(b), \ldots, \beta(b)^{n-1}$ are linearly independent over $K$, hence $V$ is a Zariski open subset of $K^{n}$.

\veno
Let $e_1,\ldots,e_n \in \Zz\left[ \overline{x} \right]$ be the elementary symmetric polynomials in $n$ variables, i.e.
$$ e_k(\overline{x}) = \sum_{1 \le i_1 < i_2 < \ldots < i_k \le n} x_{i_1} \cdots x_{i_k}. $$

Given $b = (b_{0},\ldots,b_{n-1}) \in K^{n}$ we let $$G(b) = (-e_{1}(\beta_{1}(b),\ldots,\beta_{n}(b)), e_{2}(\beta_{1}(b),\ldots,\beta_{n}(b)),\ldots, (-1)^{n} e_{n}(\beta_{1}(b),\ldots,\beta_{n}(b))).$$


\begin{claim}
\label{claim:symmtric}
There are $G_1,\ldots,G_n \in K[\overline{x}]$ such that $G(b) = (G_1(b),\ldots,G_n(b))$ for all $b \in K^n$, and if $b \in V$ then $G(b) \approx a$.
\end{claim}

The first claim of Claim~\ref{claim:symmtric} follows as $G$ is symmetric in $\alpha_{1},\ldots,\alpha_{n}$.
The second claim follows as $p_{G(b)}$ is the monic polynomial with roots $\beta_1(b),\ldots,\beta_n(b)$.
Claim~\ref{claim:jaco} below is crucial.




\begin{claim}
\label{claim:jaco}
$G(0,1,0, ..., 0) = a$ and the Jacobian of $G$ at $(0,1,0,..,0)$ is invertible.
\end{claim}

Given a polynomial function $f \colon K^n \to K^n$ we let $\jac_f(a)$ be the Jacobian of $f$ and $|\jac_f(a)|$ be the Jacobian determinant of $f$ at $a \in K^n$.

\begin{proof}
It is clear that $G(0,1,0,..,0) = a$ and  $(0,1,0,..,0)\in V$.
Let $L = K(\alpha)$.
To show that the Jacobian of $G$ at $(0,1,0,..,0)$ is invertible we first produce maps $D,E,F \colon L^n \to L^n$ such that $G$ agrees with the restriction of $D \circ E \circ F$ to $V$.
We define $F \colon L^n \to L^n$ by
$$ F(b_0,\ldots,b_{n - 1}) = (b_0 + b_1 \alpha_1 + \cdots + b_{n - 1}\alpha_1^{n - 1},\ldots, b_0 + b_1 \alpha_n + \cdots + b_{n - 1}\alpha_n^{n - 1}), $$


$E \colon L^n \to L^n$ is given by
$$ E(b) = (e_1(b),\ldots,e_n(b)), $$
and  $D \colon L^n \to L^n$ is given by 
$$ D(b_0,\ldots,b_{n - 1}) = (-b_0,b_1,-b_2,\ldots,(-1)^n b_{n - 1}). $$
So if $b \in V$ then $G(b) = (D \circ E \circ F)(b)$.
Note that $F$ and $D$ are linear, so $\jac_F$ and $\jac_D$ are constant.
Applying the chain rule we have
\begin{align*}
\jac_G(0,1,0,\ldots,0) &= \jac_D \jac_E(F(0,1,0,\ldots,0)) \jac_F \\ &= \jac_D \jac_E(\alpha_1,\ldots,\alpha_n) \jac_F.
\end{align*}
It is clear that $|\jac_D| \in \{-1,1\}$.
Furthermore $\jac_F$ is a Vandermonde matrix
\begin{equation*}
\begin{pmatrix}
1 & \alpha_1 & \alpha^2_1 & \ldots & \alpha^{n-1}_1 \\
1 & \alpha_2 & \alpha^2_2 & \ldots & \alpha^{n-1}_2 \\
1 & \alpha_3 & \alpha^2_3 & \ldots & \alpha^{n-1}_3 \\
\vdots & \vdots & \vdots & \ldots & \vdots \\
1 & \alpha_n & \alpha^2_n & \ldots & \alpha^{n-1}_n \\
\end{pmatrix}.
\end{equation*}
So $\jac_F$ is invertible as $\alpha_1,\ldots,\alpha_n$ are distinct.
Finally by  \cite{jacobian-poly}
$$ |\jac_E(\alpha_1,\ldots,\alpha_n)| = \prod_{1 \le i < j \le n} (\alpha_i - \alpha_j). $$
This is non-zero as the $\alpha_i$ are distinct, so $\jac_E(\alpha_1,\ldots,\alpha_n)$ is invertible.
\end{proof}

We now deduce Theorem ~\ref{thm:krasner}.
Let $O$ be the open subvariety of $\Aa^n$ given by $|\jac_G(\overline{x})| \ne 0$.
So $G$ gives an \'etale morphism $O \to \Aa^n$.
Then $O(K) \cap V$ is a Zariski open subset of $K^n$, which is nonempty by Claim 4.4. 
Let $W$ be an open subvariety of $\Aa^n$ such that $W(K) = O(K) \cap V$.
The restriction of $G$ to $W$ is an \'etale morphism $W \to \Aa^n$.
Let $X = G(W(K))$.
So $X$ is a nonempty EE subset of $K^{n}$ contained in the $\approx$-class of $a$. As $a$ was an arbitrary member of $U$, this concludes the proof of 
 Theorem ~\ref{thm:krasner}.

\vspace{5mm}
\noindent
Finally we can complete  the proof of Theorem~\ref{thm:bounded}.

\begin{proof}
Let $T$ be a simple theory, $M$ be a model of $T$, and $K$ be an infinite field definable in $M$.
As remarked at the beginning of this section, it suffices to fix $n$ and show that $K$ has only finitely many separable extensions of degree $n$,  and thus that the definable equivalence relation $\approx$ on the definable set $U\subset K^{n}$  has only finitely many classes. After possibly passing to an elementary extension we may suppose that $M$ is highly saturated.
By Theorem ~\ref{thm:krasner} and  Proposition~\ref{prop:generic-open}, every $\approx$-class is $f$-generic for $(K^{n},+)$.
By Lemma~\ref{lem:equiv-rel} there are only finitely many $\approx$-classes. 
\end{proof}

\subsection{Another proof that large stable fields are separably closed}
\label{section:another-proof}
We give a proof that large stable fields are separably closed that avoids Macintyre's Galois theoretic argument.
We first prove Lemma~\ref{lem:separable-ee}.
We continue to use the notation of the previous section.

\begin{lemma}
\label{lem:separable-ee}
Let $Y$ be the set of $a \in K^n$ such that $p_a$ has $n$ distinct roots in $K$.
Then $Y$ is an EE subset of $K^n$.
\end{lemma}

\begin{proof}
Let $V$ be the open subvariety of $\Aa^n$ given by $x_i \ne x_j$ for all $1 \le i < j \le n$.
Let $H \colon K^n \to K^n$ be given by $H(b) = (-e_1(b),e_2(b),\ldots,(-1)^n e_n(b))$.
So $p_{H(a)}$ is the polynomial with roots $a_1,\ldots,a_n$ for any $a = (a_1,\ldots,a_n) \in V(K)$.
It follows from \cite{jacobian-poly} that $|\jac_H(a)|$ agrees up to sign with $\prod_{1 \le i < j \le n} (a_i - a_j)$ for any $a = (a_1,\ldots,a_n) \in K^n$.
So $\jac_H(a)$ is invertible for all $a \in V(K)$.
Thus $H(V(K))$ is an EE subset of $K^n$.
\end{proof}

We now show that a large stable field is separably closed.

\begin{proof}
Suppose that $K$ is large and not separably closed.
Fact~\ref{fact:generic-basic} and Lemma~\ref{fact:generic-product} together show that if $K$ is stable then for each $n \ge 1$ there is a unique $n$-ary type over $K$ which is generic for $(K^n,+)$.
It follows by Proposition~\ref{prop:generic-open} that if $K$ is stable then any two nonempty EE subsets of $K^n$ have nonempty intersection.
As $K$ is not separably closed there is a separable, irreducible, and non-constant $p \in K[x]$.
Suppose that $p$ is monic and fix $a \in K^n$ such that $p = p_a$.
By Theorem~\ref{thm:krasner} there is an EE subset $X$ of $K^n$ such that $a \in X$ and $p_b$ is separable and irreducible for any $b \in X$.
Let $Y$ be the set of $b \in K^n$ such that $p_b$ has $n$ distinct roots in $K$, by Lemma~\ref{lem:separable-ee} $Y$ is an EE subset of $K^n$.
So $X,Y$ are disjoint nonempty EE subsets of $K^n$, hence $K$ is unstable.
\end{proof}

The proof above easily adapts to show that an infinite superstable field is algebraically closed.
We describe this proof, assuming some familiarity with superstability.
We let $\dim_U Z$ be the $U$-rank of a definable set $Z$.
Suppose that $K$ is infinite and superstable.
A superstable field is perfect, so it suffices to show that $K$ is separably closed.
Suppose otherwise and fix $n$ such that there is a nonconstant separable irreducible $p \in K[x]$.
Let $X,Y$ be as in the proof above.
Note that both $X$ and $Y$ contain a set of the form $f(W(K))$ where $W$ is a dense open subvariety of $\Aa^n$ and $f \colon W \to \Aa^n$ is \'etale.
So $\dim_U W(K) = \dim_U K^n$ and the induced map $W(K) \to K^n$ has finite fibers as $f$ is \'etale.
Hence $\dim_U X = \dim_U K^n = \dim_U Y$.
So $X,Y$ are both $f$-generic in $(K^n,+)$, which contradicts uniqueness of generic types.

\subsection{Topological corollaries}
\noindent
Suppose that $v$ is a non-trivial Henselian valuation on $K$.
It follows from the classical Krasner's lemma that each $\approx$-class is open in the $v$-adic topology on $K^n$.
See for example \cite[3.5.13.2]{poonen-qpoints} for a treatment of the case when $K$ is a local field which easily generalizes to the Henselian case.
It is shown in \cite{JTWY} that if $K$ is not separably closed then the $v$-adic topology on each $K^n$ agrees with the \'etale open topology.
So Corollary~\ref{cor:krasner} generalizes this consequence of Krasner's lemma.

\begin{corollary}
\label{cor:krasner}
Fix $a \in K^n$ such that $p_a$ is separable and irreducible.
Then the set of $b \in K^n$ such that $K[x]/(p_b)$ is $K$-algebra isomorphic to $K[x]/(p_a)$ is an \'etale open neighbourhood of $a$.
So the set of $a \in K^n$ such that $p_a$ is separable and irreducible is \'etale open.
\end{corollary} 

Fact~\ref{fact:HD} is proven in \cite{JTWY} by an application of Macintyre's Galois-theoretical argument.

\begin{fact}
\label{fact:HD}
If $K$ is not separably closed then the \'etale open topology on $K$ is Hausdorff.
\end{fact}

If $K$ is separably closed then the \'etale open topology agrees with the Zariski topology on $V(K)$ for any $K$-variety $V$, equivalently every EE subset of $V(K)$ is Zariski open.
We give a proof of Fact~\ref{fact:HD} which avoids Galois theory.
We apply the fact that if $V \to W$ is a morphism between $K$-varieties then the induced map $V(K) \to W(K)$ is \'etale open continuous.

\begin{proof}
Equip $K$ with the \'etale open topology.
Any affine transformation $x \mapsto ax + b$, $a \in K^*,b \in K$ gives a homeomorphism $K \to K$.
Thus it is enough to produce two disjoint nonempty \'etale open subsets of $K$.
The argument of Section~\ref{section:another-proof} yields two disjoint nonempty \'etale open subsets $X,Y$ of $K^n$.
Fix $p \in X$ and $q \in Y$ and let $f \colon K \to K^n$ be given by $f(t) = (1 - t)p + tq$.
Then $f$ is a continuous map between \'etale open topologies so $f^{-1}(X), f^{-1}(Y)$ are disjoint nonempty \'etale open subsets of $K$.
\end{proof}

Finally, we characterize bounded $\mathrm{PAC}$ fields amongst $\mathrm{PAC}$ fields.

\begin{corollary}
\label{cor:bounded-char}
Suppose that $K$ is $\mathrm{PAC}$, equip each $K^n$ with the \'etale open topology.
Then $K$ is bounded if and only if any definable equivalence relation on $K^n$ has only finitely many classes with interior.
\end{corollary}

Note that Corollary~\ref{cor:bounded-char} fails when ``$\mathrm{PAC}$" is replaced by ``large".
For example $\mathbb{Q}_p$ is bounded, the \'etale open topology on $\mathbb{Q}_p$ agrees with the $p$-adic topology, and the equivalence relation where $E(a,b)$ if and only if $a,b \in \mathbb{Q}_p$ have the same $p$-adic valuation is definable and has infinitely many open classes.

\begin{proof}
Suppose that $K$ is not bounded.
Fix $n$ such that $K$ has infinitely many separable extensions of degree $n$.
Let $U$ and $\approx$ be as in the proof of Theorem~\ref{thm:bounded}.
Then each $\approx$-class is open and there are infinitely many $\approx$-classes.
Now suppose that $K$ is bounded and $E$ is a definable equivalence relation on $K^n$.
Note that $K$ is simple.
By Proposition~\ref{prop:generic-open} any $E$-class with interior is $f$-generic.
The proof of Lemma~\ref{lem:equiv-rel} shows that there are only finitely many $f$-generic $E$-classes.
\end{proof}

\section{Additional remarks and results}
We will discuss a few related topics and results, and prove Theorem~\ref{thm:projective}.
If $\Chara(K) = p > 0$ then we let $\wp \colon K \to K$ be the \textbf{Artin-Schreier map} $\wp(x) = x^p - x$.
This map is an additive homomorphism, so $\wp(K)$ is a subgroup of $(K,+)$.
In this section we let $P_n = \{ a^n : a \in K^*\}$ for each $n$.
Some of our proofs below could be simplified by apply Scanlon's theorem \cite{Kaplan2011} that an infinite stable field is Artin-Schreier closed, but we will avoid this.

\subsection{Boundedness and large stable fields}
\label{section:stable}
It is a theorem of Poizat that an infinite bounded stable field is separably closed.
Poizat's result and Theorem~\ref{thm:bounded} together show that large stable fields are separably closed.
Poizat's result is mentioned somewhat informally at the bottom of p. 347 in \cite{Poizat-generic} and does not appear to be well-known, so we take the opportunity to clarify the matter.
Fact~\ref{fact:poizat-lem} is \cite[Lemma 4]{Poizat-generic}.

\begin{fact}
\label{fact:poizat-lem}
Suppose that $L$ is a finite Galois extension of $K$.
Then the following holds.
\begin{enumerate}
\item If $q \ne \Chara(K)$ is a prime then there are only finitely many cosets $H$ of $P_q$ in $(K^*,\times)$ such that some (equivalently: any) $a \in H$ has a $q$th rooth in $L$.
\item Suppose  $\Chara(K) = p > 0$.
Then there are only finitely many cosets $H$ of $\wp(K)$ in $(K,+)$ such that some (equivalently: any) $a \in H$ is of the form $b^p - b$ for some $b \in L$.
\end{enumerate}
\end{fact}

Fact~\ref{fact:bounded} follows from Fact~\ref{fact:poizat-lem}.

\begin{fact}
\label{fact:bounded}
Suppose that $K$ is bounded.
Then
\begin{enumerate}
\item if $q \ne \Chara(K)$ is prime then $P_q$ has finite index in $(K^*,\times)$, and
\item if $\Chara(K) > 0$ then $\wp(K)$ has finite index in $(K,+)$.
\end{enumerate}
\end{fact}

We sketch a proof.
See \cite[Lemma 2.2]{FJ-bounded} for a proof of the characteristic zero case of Fact~\ref{fact:bounded}.1 via Galois cohomology.

\begin{proof}
We only prove $(1)$ as the proof of $(2)$ is similar.
Suppose $a \in K^*$ and $\alpha \in K^{\mathrm{sep}}$ satisfies $\alpha^q = a$.
Then $\alpha$ and its conjugates generate a degree $\le q$ Galois extension of $K$.
As $K$ is bounded there are only finitely many such extensions.
So by Fact~\ref{fact:bounded} $P_q$ has finite index in $(K^*,\times)$.
\end{proof}

Finally, Fact~\ref{fact:galois} is essentially proven in \cite{Macintyre-omegastable} via a Galois-theoretic argument.

\begin{fact}
\label{fact:galois}
Suppose that the following holds for any finite Galois extension $L$ of $K$:
\begin{enumerate}
\item the $q$th power map $L^* \to L^*$ is surjective for any prime $q \ne \Chara(K)$, and
\item if $\Chara(K) \ne 0$ then the Artin-Schreier map $L \to L$ is surjective.
\end{enumerate}
Then $K$ is separably closed.
\end{fact}

We now sketch a proof of Poizat's theorem.

\begin{corollary}
Suppose that $K$ is infinite, bounded, and stable.
Then $K$ is separably closed.
\end{corollary}

\begin{proof} 
We verify the conditions of Fact~\ref{fact:galois}.
Suppose that $L$ is a finite Galois extension of $K$.
Then $L$ is bounded and stable (the latter holds as $L$ is interpretable in $K$).
As $L$ is stable there is a unique additive (multiplicative) generic type over $K$, see Fact~\ref{fact:generic-basic}.
It follows that there are no proper finite index definable subgroups of $(L^*,\times)$ or $(L,+)$.
So by Fact~\ref{fact:bounded} the $q$th power map $L^* \to L^*$ is surjective for any prime $q \ne \Chara(K)$ and if $\Chara(K) > 0$ then the Artin-Schreier map $L \to L$ is surjective.
\end{proof}

We  repeat that it follows from Fact 5.1 and Theorem~\ref{thm:bounded} that: 

\begin{corollary} 
\label{cor:poizat-large}
Suppose that $K$ is large and simple.
Then
\begin{enumerate}
\item if $q \ne \Chara(K)$ is prime then $P_q$ has finite index in $(K^*,\times)$, and
\item if $\Chara(K) > 0$ then $\wp(K)$ has finite index in $(K,+)$.
\end{enumerate}
\end{corollary}

Corollary~\ref{cor:poizat-large}.2 is proven more generally for infinite simple fields in \cite{Kaplan2011}.
We take the opportunity to sketch a direct proof of Corollary~\ref{cor:poizat-large}.
We let $\Gg_m$ be the scheme-theoretic multiplicative group $\Spec K[x,x^{-1}]$, so $\Gg_m(K) = K^*$.

\begin{proof}
We first fix a prime $q \ne \Chara(K)$.
The morphism ${\mathbb G}_{m} \to \Aa^{1}$ given by $x \mapsto ax^q$ is \'etale for any $a \in K^*$.
So any coset of $P_q$ is an EE subset of $K$.
By the special (and easier) case of  Proposition~\ref{prop:generic-open} when $n=1$, any coset of $P_q$ is $f$-generic in $(K^*,\times)$.
By Lemma~\ref{lem:index-generic} $P_q$ has finite index in $(K^*,\times)$.
Item $(2)$ follows by a similar argument and the fact that the Artin-Schreier morphism $\Aa^1 \to \Aa^1$ is \'etale.
\end{proof}

Fehm and Jahnke construct an unbounded $\mathrm{PAC}$ field $K$ such that the group of $n$th powers has finite index in each finite extension of $K$~\cite[Proposition 4.4]{FJ-bounded}, so Theorem~\ref{thm:bounded} does not follow from Corollary~\ref{cor:poizat-large}.

\subsection{Conics, Brauer group, and projectivity}
\label{section:conics}
Corollary~\ref{cor:power-sum} follows from Corollary~\ref{cor:poizat-large} and Fact~\ref{fact:coset-sum}.

\begin{corollary}
\label{cor:power-sum}
Suppose that $K$ is large and simple, $a,b \in K^*$, and $p \ne \Chara(K)$ is a prime.
Then there are $c,d \in K$ such that $c^p + ad^p = b$.
\end{corollary}
The proof in ~\cite{supersimple-field} that conics over (infinite) supersimple fields have points now extends to proving Theorem ~\ref{thm:conics}.

\veno
{\em Proof of  Theorem \ref{thm:conics}.}
Let $C$ be a conic, i.e. a smooth projective irreducible $K$-curve of genus $0$.
As $\Chara(K) \ne 2$ we may assume that $C$ is a closed subvariety of $\mathbb{P}^2$ given by the homogenous equation $ax^2 + by^2 = z$ for some $a,b \in K^*$.
By Corollary~\ref{cor:power-sum} there are $c,d \in K$ such that $ac^2 + bd^2 = 1$.
So $C(K)$ is nonempty.

\veno
We let $\br K$ be the Brauer group of $K$.
Recall that the Brauer group of an arbitrary field is an abelian torsion group.
Given a prime $p$ we let $\br_p K$ be the $p$-part of the Brauer group of $K$.
Facts~\ref{fact:psw} and \ref{fact:psw-1} both follow by the proof of \cite[Theorem 4.6]{supersimple-field}.

\begin{fact}
\label{fact:psw}
Let $p \ne \Chara(K)$ be a prime.
Suppose that whenever $L$ is a finite separable extension of $K$ and $a \in L^*$, then $\{ b^p + ac^p : b,c \in L^* \}$ contains $L^*$. 
Then $\br_p K$ is trivial.
\end{fact}

\begin{fact}
\label{fact:psw-1}
Suppose that:
\begin{enumerate}
\item $K$ is perfect, and
\item if $L$ is a finite extension of $K$, $p$ is a prime, and $a \in L^*$, then $\{ b^p + ac^p : b,c \in L^*\}$ contains $L^*$.
\end{enumerate}
Then the Brauer group of $K$ is trivial.
\end{fact}

We now prove Theorem~\ref{thm:brauer}.

\begin{proof}
It suffices to show that the second condition of Fact~\ref{fact:psw-1} is satisfied.
Let $L$ be a finite extension of $K$ and $p$ be a prime.
Note that $L$ is perfect as a finite extension of a perfect field is perfect, the case when $p = \Chara(K)$ follows.
Suppose that $p \ne \Chara(K)$.
Note that $L$ is simple as $L$ is interpretable in $K$ and $L$ is large by Fact~\ref{fact:alg-ext}.
Apply Corollary~\ref{cor:power-sum}.
\end{proof}

\noindent
Theorem~\ref{thm:projective} follows from Proposition~\ref{prop:coho} as a field of cohomological dimension $\le 1$ has projective absolute Galois group, see Gruenberg~\cite{gruenberg}.

\begin{proposition}
\label{prop:coho}
If $K$ is simple and large then $K$ has cohomological dimension $\le 1$.
\end{proposition}

See \cite[I \textsection 3]{serre-coho} for an overview of cohomological dimension.
The proof of Proposition~\ref{prop:coho} is due to Philip Dittmann.
We do not know if every large simple field has trivial Brauer group.

\begin{proof}
Suppose $K$ is simple and large.
The same argument as in the proof of Theorem~\ref{thm:brauer} shows that if $p \ne \Chara(K)$ is a prime, $L$ is a finite separable extension of $K$, and $a \in L^*$, then $\{ b^p + ac^p : b,c \in L^*\}$ contains $L^*$.
So by Fact~\ref{fact:psw} $\br_p L$ is trivial for every finite extension $L$ of $K$ and prime $p \ne \Chara(K)$.
By \cite[II.2.3 Proposition 4]{serre-coho} $K$ has $p$-cohomological dimension $\le 1$ for every prime $p \ne \Chara(K)$.
By \cite[II.2.2 Proposition 3]{serre-coho} any field $L$ has $\Chara(L)$-cohomological dimension $\le 1$.
So $K$ has cohomological dimension $\le 1$.
\end{proof}


\appendix
\section{$\nsopinfty$ fields}

A theory $T$ has the \textbf{fully finite strong order property} if there is a formula $\psi(x,y)$, with the two tuples of variables $x$ and $y$ having equal length, a model $M \models T$, and a sequence $(a_i)_{i \in \omega}$ of tuples in $M$ satisfying $M \models \psi(a_i, a_j)$ for all $i < j$, and for any $n \geq 3$ the formula $\psi(x_1, x_2) \land \dotsb \wedge \psi(x_{n-1}, x_n) \wedge \psi(x_n, x_1)$ is inconsistent with $T$.
In short, the binary relation described by $\psi$ admits infinite chains and does not admit cycles.
In this situation we also say that $T$ has or is $\sopinfty$.
A structure $M$ is $\sopinfty$ if its theory is.
A theory or structure is $\nsopinfty$ if it is not $\sopinfty$.

\meno
In Shelah's terminology, $T$ having the fully finite strong order property witnessed by $\psi$ is equivalent to $\psi$ having the $n$-strong order property $\mathrm{SOP}_n$ for $T$, for all $n \geq 3$ \cite[Definition~2.5]{shelah-sop}.
In particular, if $T$ is complete and simple then $T$ is $\nsopinfty$ \cite[Claim~2.7]{shelah-sop}.
The notion ``fully finite strong order property'' seems to have first appeared in an unpublished manuscript by Adler \cite{adler-unpub}, although it has by now also been used elsewhere \cite[Definition~2.1]{ConantTerry_Urysohn}.

\subsection{Valuations}

The Henselization of a PAC field with respect to any non-trivial valuation is separably closed \cite[Corollary 11.5.9]{field-arithmetic}.
Thus the following can be seen as supporting evidence for the conjecture that large simple fields are PAC.

\begin{theorem}  
\label{thm:largevalued}
Suppose that $K$ is large and $v$ is a non-trivial valuation on $K$.
If $(K,v)$ has non-separably closed Henselization then $K$ is $\sopinfty$.
In particular if either the residue field of $v$ is not algebraically closed or the value group of $v$ is not divisible then $K$ is $\sopinfty$.
\end{theorem}

The second claim of Theorem~\ref{thm:largevalued} follows from the first as the Henselization of $(K,v)$ has the same residue field and value group as $(K,v)$ \cite[Theorem 5.2.5]{EP-value}, and a non-trivially valued separably closed field has algebraically closed residue field and divisible value group~\cite[Theorem 3.2.11]{EP-value}.
We will make use of Fact~\ref{fact:hensel-refine}, proven in \cite[Theorem 6.15]{JTWY}.
 
\begin{fact}
\label{fact:hensel-refine} Let $v$ be a non-trivial valuation on $K$.
If the Henselization of $(K,v)$ is not separably closed then the \'etale open topology refines the $v$-adic topology on $K$.
\end{fact}

The following argument using generics was used in a preliminary version of the main article to get the simple case of Theorem~\ref{thm:largevalued}.
Suppose that $K$ is simple and the Henselization of $(K,v)$ is not separably closed.
By Fact~\ref{fact:hensel-refine}, $\mfrakv$ is an \'etale open neighbourhood of $0$, so there is an EE subset $U$ of $K$ satisfying $0 \in U \subset \mfrakv$.
By Proposition~\ref{prop:generic-open} the set $U$ is $f$-generic for $(K,+)$.
This contradicts Lemma~\ref{lem:index-generic} as $\mfrakv$ is an infinite index subgroup of $(K,+)$.
This argument does not generalize to large $\nsopone$ fields as at present there is no theory of generics in $\mathrm{NSOP}_1$ groups.
(This is not straightforward: Dobrowolski~\cite{dobrowolski2020sets} gives an example of a definable group in an $\mathrm{NSOP}_1$ structure in which generics with respect to Kim forking do not exist.)

\begin{lemma}
\label{lem:powers}
Suppose that $K$ is large and $U \subseteq K$ is an \'etale open neighbourhood of zero.
Then for any $n \ge 2$ there is $a \in K^*$ such that $a,a^2,\ldots,a^n \in U$.
\end{lemma}

\begin{proof}
For each $i \in \{2,\ldots,n\}$ let $V_i = \{ b \in K : b^i \in U\}$.
Each map $K \to K$, $b \mapsto b^i$ is continuous with respect to the \'etale open topology, so each $V_i$ is an \'etale open neighbourhood of zero.
Then $V = V_1 \cap \ldots \cap V_n$ is an \'etale open neighbourhood of zero.
As $K$ is large $V$ contains a nonzero element of $K$.
\end{proof}

\begin{proof}[Proof of Theorem~\ref{thm:largevalued}]
By Fact~\ref{fact:hensel-refine} there is a nonempty EE subset $U$ of $K$ with $0 \in U \subseteq \mfrakv$.
Let $\psi(x,y)$ be the formula $(x \ne 0) \land (y \ne 0) \land (x^{-1}y \in U)$.
Note that if $K \models \psi(a,b)$ then $b/a \in \mfrakv$, hence $v(a) < v(b)$.
We show that $\psi(x,y)$ witnesses $\sopinfty$.
First suppose that $a_1,\ldots,a_n \in K$ satisfy
\[
\psi(a_1,a_2) \land \cdots \land \psi(a_{n -1},a_n) \land \psi(a_n,a_1).
\]
Then we have $v(a_1) < v(a_2) < \cdots < v(a_{n - 1}) < v(a_n) < v(a_1)$, contradiction.
We now show that for each $n \ge 1$ there are $a_1,\ldots,a_n \in K$ such that $K \models \psi(a_i,a_j)$ if and only if $i < j$.
By Lemma~\ref{lem:powers} there is $a \in K^*$ such that $a,a^2,\ldots,a^n \in U$.
Then $K \models \psi(a^i,a^j)$ for $i < j$.
Thus the binary relation on $K$ defined by $\psi$ admits chains of arbitrary finite length, and in a saturated elementary extension of $K$ we obtain an infinite chain.
Thus $K$ is $\sopinfty$.
\end{proof}

\meno
Recall that EE sets are existentially definable.
Note that the formula $\psi$ in the proof of Theorem~\ref{thm:largevalued} is existential.
This is optimal as a quantifier-free formula in an arbitrary field is stable.
This is similar to the result, proven in \cite[Theorem 3.1]{JTWY}, that an unstable large field admits an unstable existential formula.
The witnesses for $\sopinfty$ produced in the next section are also existential.

\meno

If $v$ is actually Henselian, the same technique as in the proof of Theorem~\ref{thm:largevalued} gives a slightly stronger statement.
This is presumably well known to the experts, but appears not to be available in the literature.
\begin{theorem}
Suppose that $v$ is a non-trivial Henselian valuation on $K$ and $K$ is not separably closed.
Then $K$ has the strict order property \cite[Definition 2.1]{shelah-sop}.
\end{theorem}
\begin{proof}
Fact~\ref{fact:hensel-refine} provides an EE subset $U$ of $K$ with $0 \in U \subseteq \mathfrak{m}_v$.
By \cite[Theorem B]{JTWY} the $v$-topology on $K$ agrees with the étale open topology, hence $U$ is $v$-open, and in particular contains a ball around $0$.\footnote{This
argument does not seriously use the étale open topology -- we only need that the topology given by $v$ is definable in the field language.
This latter fact is already implicit in \cite[Remark 7.11]{Prestel1978}.
  }
Therefore, for any element $c \in K^\times$ with $v(c)$ sufficiently large we have $cU \subsetneq U$.
Thus the definable family $\{ xU \colon x \in K \}$ contains the infinite chain $U \supsetneq cU \supsetneq c^2 U \supsetneq \dotsb$ under inclusion.
Hence $K$ has the strict order property.
\end{proof}



\subsection{Formally real and formally $p$-adic fields}
\label{section:formally real}
Corollary~\ref{cor:power-sum} implies that if $K$ is large, simple, and of characteristic zero then there are $a,b \in K$ such that $a^2 + b^2 = -1$.
So a large simple field cannot be formally real.
Duret~\cite{duret-formally-real} showed that formally real fields are unstable.
Theorem~\ref{thm:formally real} generalizes these.

\begin{theorem}
\label{thm:formally real}
Suppose that $K$ is formally real.
Then $K$ is $\sopinfty$.
\end{theorem}

\begin{proof}
Let $\varphi(x,y)$ be the formula
$$ \exists z_1,z_2,z_3,z_4 \left[ x - y - 1 = z^2_1 + z^2_2 + z^2_3 + z^2_4\right]. $$
We show that $\varphi$ witnesses $\sopinfty$.
An application of Lagrange's four-square theorem shows that $K \models \varphi(m,m')$ for all integers $m > m'$.
Now suppose that $a_1,\ldots,a_n \in K$ and we have $K \models \{ \varphi(a_1,a_2),\ldots,\varphi(a_{n - 1},a_n),\varphi(a_n,a_1) \}$.
Then 
$$ -n = (a_1 - a_2 - 1) + (a_2 - a_3 - 1) + \dotsb + (a_{n - 1} - a_n - 1) + (a_n - a_1 - 1) $$
is a sum of squares, contradiction.
\end{proof}

Fix a prime $p$.
A field $K$ is $p$-adically closed if $K$ is elementarily equivalent to a finite extension of $\Qq_p$ and $K$ is \textbf{formally $p$-adic} if $K$ embeds into a $p$-adically closed field.
An equivalent definition (which we shall not need) is that there exists a $p$-valuation $v$ on $K$, i.e.\ $v$ is of mixed characteristic, the residue field is a finite extension of $\mathbb{F}_p$, and the interval $[0, v(p)]$ in the value group is finite.
Indeed, if $v$ is a $p$-valuation on $K$ then the so-called $p$-adic closure of $(K,v)$ is an elementary extension of a finite extension of $\Qq_p$.
See \cite{Prestel-roquette} for a comprehensive treatment of formally $p$-adic fields.

\begin{theorem}
  \label{thm:formally-p-adic}
  Suppose that $K$ is formally $p$-adic.
  Then $K$ is $\sopinfty$.
\end{theorem}




\begin{proof}
Let $F$ be a finite extension of $\Qq_p$ such that $K$ embeds into an elementary extension of $F$.
Let $v$ be the unique extension of the $p$-adic valuation on $\mathbb{Q}_p$ to $F$ and $\mathcal{O}_F$ be the valuation ring of $v$.

\meno  
By \cite[Proposition 4.7 and Proposition 4.8]{AnscombeDittmannFehm-Siegel} (applied to the base field $K=\mathbb{Q}$, the prime $\mathfrak p = p$ of $\mathbb{Q}$, and the relative type $\tau$ of $F/\mathbb{Q}_p$ in the terminology there), there exists a  parameter-free existential formula $\psi(x)$ such that $\psi(F) \subseteq \mathcal{O}_F$, and $\psi(\mathbb{Q}) = \mathbb{Z}_{(p)}$.
  (Note that the paper cited phrases the result in terms of a concrete ``diophantine family'' $D^\tau_{\mathfrak p, t_{\mathfrak p}, A, B}$, but this is effectively the same as a  existential formula with parameters from the base field $\mathbb{Q}$ \cite[Remark 3.2]{AnscombeDittmannFehm-Siegel}, and parameters from $\mathbb{Q}$ can be eliminated.)

\meno
Let $\varphi(x,y)$ be the formula
  \[ (y \ne 0) \land \exists z (\psi(z) \wedge y = p \cdot x \cdot z ) .\]

We show that $\varphi(x,y)$ witnesses $\sopinfty$ for $K$.
Suppose $m < m'$ are integers.
Then we have $\mathbb{Q} \models \varphi(p^m, p^{m'})$, since
$p^{m'}/(p \cdot p^m)\in \mathbb{Z} \subseteq \psi(\mathbb{Q})$.
Since $\varphi$ is existential, we have $K \models \varphi(p^m,p^{m'})$.
Thus the binary relation on $K$ defined by $\varphi$ admits an infinite chain.

\meno
Now suppose that $K$ satisfies 
\[
\Theta = \exists x_1,\ldots,x_n [\varphi(x_1,x_2) \land \dotsb \land \varphi(x_{n - 1},x_n) \land \varphi(x_n,x_1)].
\]
As $\Theta$ is existential and $K$ embeds into an elementary extension of $F$, we have $F \models \Theta$.
Hence there are $b_1, \dotsc, b_n \in F$ such that $F \models \{ \varphi(b_1, b_2), \dotsc, \varphi(b_{n-1}, b_n), \varphi(b_n, b_1) \}$.

\meno
As $\psi(F) \subseteq \mathcal{O}_F$ we see that $F \models \varphi(a,a')$ implies that $v(a) < v(a')$ for any $a,a' \in F$.
Thus we have $v(b_1) < v(b_2) < \dotsb < v(b_{n-1}) < v(b_n) < v(b_1)$, a contradiction.
\end{proof}

\bibliographystyle{amsalpha}
\bibliography{ref}

\end{document}